\documentclass[11pt]{article}%
\usepackage[utf8]{inputenc}
\usepackage[T1]{fontenc}
\usepackage[english]{babel}
\usepackage{indentfirst}
\usepackage{latexsym}
\usepackage[colorlinks=true, bookmarksnumbered=true, bookmarksopen=true,
bookmarksopenlevel=3, pdfstartview=FitH, linkcolor=cyan, pdfmenubar=true,
pdftoolbar=true, bookmarks=true,citecolor=cyan, urlcolor=magenta,
filecolor=cyan,menucolor=black,plainpages=false,pdfpagelabels]{hyperref}
\usepackage[lmargin=2cm,tmargin=2cm,bmargin=2cm,rmargin=2cm]{geometry}
\usepackage{amsbsy, amsmath,amsfonts,amssymb,amsthm}
\usepackage{times}
\usepackage{graphicx}
\usepackage{amsmath}
\usepackage{amsfonts}
\usepackage{amssymb}%
\newtheorem{theorem}{Theorem}[section]
\newtheorem{definition}[theorem]{Definition}

\newtheorem{corollary}[theorem]{Corollary}

\newtheorem{remark}[theorem]{Remark}
\newtheorem{example}[theorem]{Example}
\numberwithin{equation}{section}
\begin{document}

\title{Banach spaces of linear operators and homogeneous polynomials without the approximation property}
\author{{{Sergio A. Pérez}{\thanks{S.Pérez was supported by CAPES and CNPq, Brazil. (corresponding author) }}}\\{\small  IMECC, UNICAMP} \\{\small {Rua S\'{e}rgio Buarque de Holanda, 651, CEP 13083-859, Campinas-SP,
Brazil.}}\\{\small \texttt{Email:Sergio.2060@hotmail.com}}\\\vspace{-0.2cm}}
\date{}
\maketitle

\begin{abstract}

We present many examples of Banach spaces of linear operators and homogeneous polynomials without the approximation property, thus  improving results of Dineen and Mujica \cite{JORGE} and Godefroy and Saphar \cite{THREE}.



\end{abstract}

\renewcommand{\baselinestretch}{1.1}

\section{Introduction}

The approximation property was introduced by Grothendieck \cite{GR}. Enflo \cite{ENFLO}
gave the first example of a Banach space without the approximation property.
Enflo’s counterexample is an artificially constructed Banach space. The first naturally
defined Banach space without the approximation property was given by
Szankowski \cite{SZ}, who proved that the space $\mathcal{L}(\ell_{2};\ell_{2} ) $ of continuous linear operators
on $\ell_{2}$ does not have the approximation property. Later Godefroy and
Saphar \cite{THREE} proved that, if $\mathcal{L}_{K}(\ell_{2}; \ell_{2})$ denotes the subspace of all compact members
of $\mathcal{L}(\ell_{2};\ell_{2} ) $, then the quotient $\mathcal{L}(\ell_{2};\ell_{2} )/\mathcal{L}_{K}(\ell_{2}; \ell_{2}) $ does not have the approximation
property.

Recently Dineen and Mujica \cite{JORGE} proved that if $1< p\leq q<\infty$, then $\mathcal{L}(\ell_{p};\ell_{q} )$ does not have the approximation property. They also proved that if $1< p<\infty$ and $n\geq p$, then the space $\mathcal{P}(^{n}\ell_{p})$ of continuous n-homogeneous polynomials on $\ell_{p}$ does not have the approximation property.

In this paper, by refinements of the methods of Dineen and Mujica \cite{JORGE} and Godefroy and Shapar \cite{THREE}, we present many examples of Banach spaces of linear operators and homogeneous polynomials which do not have the approximation property.

In Section $2$ we present some examples of Banach spaces of linear operators without the approximation property. Among other results, we show that
if $1<p\leq q<\infty$, and $E$ and $F$ are closed infinite dimensional subspaces of $\ell_{p}$ and $\ell_{q}$, respectively, then $\mathcal{L}(E; F)$ does not have the approximation property. This improves a result of Dineen and Mujica \cite{JORGE}. We also show that if $1<p\leq q<\infty$, and $E$ and $F$ are closed
infinite dimensional subspaces of $\ell_{p}$ and $\ell_{q}$, respectively, then the quotient  $\mathcal{L}(E; F)/\mathcal{L}_{K}(E; F)$ does not have the approximation property. This improves a result of Godefroy and Saphar \cite{THREE}.

In Section $3$ we present more examples of Banach spaces of linear operators without the approximation property. Our examples are Banach spaces of linear operators on Pelczynski's universal space $U_{1}$, on Orlicz sequence spaces $\ell_{M_{p}}$, and on Lorentz sequence spaces $d(w,p)$.

In Section $4$ we present examples of Banach spaces of homogeneous polynomials without the approximation property.
Among other results we show that if $1<p<\infty$ and $E$ is a closed infinite dimensional subspace of $\ell_{p}$, then $\mathcal{P}(^{n}E)$ does not have
the approximation property for every $n\geq p$. This improves another result of Dineen and Mujica \cite{JORGE}. We also show that if $1<p\leq q<\infty$, and
$E$ and $F$ are closed infinite dimensional subspaces of $\ell_{p}$ and $\ell_{q}$, respectively, then $\mathcal{P}(^{n}E; F)$ does not have the approximation
property for every $n\geq 1$. We also show that if $n<p\leq q<\infty$, and $E$ and $F$ are closed infinite dimensional subspaces of $\ell_{p}$ and $\ell_{q}$, respectively, then the quotient  $\mathcal{P}(^{n}E; F)/\mathcal{P}_{K}(^{n}E; F)$ does not have the approximation property.
\bigskip

This paper is based on part of the author's doctoral thesis at the Universidade Estadual de Campinas. This research has been supported by CAPES and CNPq.

The author is grateful to his thesis advisor, Professor Jorge Mujica, for his advice and help. He is also grateful to Professor Vinicius
Fávaro for some helpful suggestions.

\section{Banach spaces of linear operators without the approximation property}

Let $E$ and $F$ denote Banach spaces over $ \mathbb{K}$, where $ \mathbb{K}$ is $ \mathbb{R}$ or  $\mathbb{C}$. Let $E^{\prime}$
denote the dual of $E$. Let $ \mathcal{L}(E;F)$ (resp. $\mathcal{L}_{K}(E;F)$) denote  the space of all bounded (resp. compact)
linear operators from $E$ into $F$.

Let us recall that $E$ is said to have the approximation property if given $K \subset E$
compact and $\epsilon > 0$, there is a finite rank operator $T\in \mathcal{L}(E;E)$ such that $\|Tx-x\| < \epsilon$
for every $x\in K$. If $E$ has the approximation property, then every complemented
subspace of $E$ also has the approximation property. $E$ is said to have the bounded approximation property if there exists $\lambda \geq 1$ so that for every compact subset $K \subset E$ and for every $\epsilon >0$ there is a finite rank operator $T\in \mathcal{L}(E;E)$ such that $\|T\|\leq \lambda$ and $\|Tx-x\| < \epsilon$
for every $x\in K$.
$E$ is isomorphic to a complemented subspace of $F$ if and only if there are $A \in
\mathcal{L}(E; F)$ and $B\in \mathcal{L}(F;E)$ such that $B\circ A = I$. These simple remarks will be repeatedly
used throughout this paper.

\begin{theorem}\label{proposicion 1}
Let $E$ and $F$ be Banach spaces. If $E$ and $F$  contain complemented subspaces isomorphic to $M$ and $N$, respectively, then $\mathcal{L}(E;F)$ contains
a complemented subspace isomorphic to $\mathcal{L}(M;N)$.
\end{theorem}

\begin{proof}
 By hypothesis there are $A_{1}\in \mathcal{L}(M;E)$, $B_{1}\in \mathcal{L}(E;M)$, $A_{2}\in \mathcal{L}(N;F)$, $B_{2}\in \mathcal{L}(F;N)$ such that $B_{1}\circ A_{1}=I$ and $B_{2}\circ A_{2}=I$. Consider the operators
$$C: S\in \mathcal{L}(M;N)\rightarrow  A_{2}\circ S\circ B_{1}\in \mathcal{L}(E;F)$$
and
$$D: T \in \mathcal{L}(E;F)\rightarrow  B_{2}\circ T\circ A_{1}\in \mathcal{L}(M;N).$$
Then $D\circ C=I$ and the desired conclusion follows.
\end{proof}

The next theorem improves  \cite[Proposition 2.4]{JORGE}.

\begin{theorem}\label{corolario 1}

Let $1<p\leq q<\infty$. If $E$ and $F$  contain complemented subspaces isomorphic to $\ell_{p}$ and $\ell_{q}$, respectively, then $\mathcal{L}(E;F)$  does not have the approximation property.

\end{theorem}

\begin{proof}
 By Theorem \ref{proposicion 1} $\mathcal{L}(E;F)$ contains a complemented subspace isomorphic to $\mathcal{L}(\ell_{p};\ell_{q})$. Then the conclusion follows from  \cite[Proposition 2.2 ]{JORGE}.
 \end{proof}

The next result improves  \cite[Proposition 2.2]{JORGE}.

\begin{theorem} \label{thm:(Teorema 6)}
 Let $1<p\leq q<\infty$, and let $E$ and $F$ be closed infinite dimensional subspaces of $\ell_{p}$ and $\ell_{q}$, respectively. Then $\mathcal{L}(E;F)$
 does not have the approximation property.
 \end{theorem}

\begin{proof}
 By \cite[Lemma 2 ]{Pel} or \cite[Proposition 2.a.2]{LIN} $E$ and $F$ contain complemented subspaces  isomorphic to $\ell_{p}$ and $\ell_{q}$, respectively. Then the desired conclusion follows from Theorem \ref{corolario  1}.
\end{proof}

The next result complements \cite[Proposition 2.1]{JORGE}.

\begin{theorem}\label{abc}
 Let $2<p\leq q<\infty$, and let  $E$ and $ F$ be closed infinite dimensional subspaces of  $L_{p}[0,1]$ and $L_{q}[0,1]$, respectively, with  $F$  not isomorphic to $\ell_{2}$. Then $\mathcal{L}(E;F)$  does not have the approximation property.
\end{theorem}

\begin{proof}
\begin{enumerate}

\item [(i)] If $E$  is not isomorphic to $\ell_{2}$, then it follows from \cite[p. 206, 12.5]{BANACH} that $E$ and $F$ contain complemented subspaces isomorphic to $\ell_{p}$ and $\ell_{q}$, respectively. Then the desired conclusion follows  from Theorem \ref{corolario 1}.

\item [(ii)] If $E$ is isomorphic to $\ell_{2}$,  then the same argument shows that $\mathcal{L}(E;F)$ contains
a complemented subspace isomorphic to  $\mathcal{L}(\ell_{2}; \ell_{q})$, and the desired conclusion follows as before.

\end{enumerate}
\end{proof}

The next result improves \cite[Corollary 2.8]{THREE}.

 \begin{theorem}\label{proposicion 9}
If $1<p\leq q<\infty$, then $\mathcal{L}(\ell_{p};\ell_{q})/\mathcal{L}_{K} (\ell_{p};\ell_{q})$ does not have the approximation property.
\end{theorem}

\begin{proof}
 We apply \cite[Theorem 2.4]{THREE}. By \cite[ Lemma 1]{REMARKS}   $\mathcal{L}_{K}(\ell_{p};\ell_{q})^{\bot}$ is a complemented subspace of $\mathcal{L}(\ell_{p};\ell_{q})^{\prime}$. By \cite[ Corollary 16.69]{FABIAN} $\mathcal{L}_{K}(\ell_{p};\ell_{q})$ has a Schauder basis. If we assume that $\mathcal{L}(\ell_{p};\ell_{q})/\mathcal{L}_{K} (\ell_{p};\ell_{q})$ has the approximation property, then  \cite[ Theorem 2.4]{THREE} would imply that  $\mathcal{L}(\ell_{p};\ell_{q})$ has the approximation property, thus contradicting \cite[Proposition 2.2]{JORGE}.
\end{proof}

\begin{remark}
 If $1< q< p<\infty$, then  $\mathcal{L}(\ell_{p};\ell_{q})=\mathcal{L}_{K}(\ell_{p};\ell_{q})$, by a result of Pitt \cite{Pitt}. Hence the restriction $p\leq q$  in the preceding theorem cannot be deleted.
\end{remark}

 \begin{theorem}\label{proposicion 9.1}
 If $E$ and $F$ contain complemented subspaces isomorphic to $M$ and $N$, respectively, then $\mathcal{L}(E;F)/\mathcal{L}_{K} (E;F)$ contains a complemented
 subspace isomorphic to $\mathcal{L}(M;N)/\mathcal{L}_{K} (M;N)$.
\end{theorem}

\begin{proof}
By hypothesis there are
$A_{1}\in \mathcal{L}(M;E)$, $B_{1}\in \mathcal{L}(E;M)$, $A_{2}\in \mathcal{L}(N;F)$, $B_{2}\in \mathcal{L}(F;N)$ such that $B_{1}\circ A_{1}=I$ and $B_{2}\circ A_{2}=I$.
Let $C: \mathcal{L}(M;N)\rightarrow  \mathcal{L}(E;F) $ and $D: \mathcal{L}(E;F)\rightarrow  \mathcal{L}(M;N)$ be the operators from the proof of Theorem
\ref{proposicion 1}. Since $C(\mathcal{L}_{K} (M;N))\subset \mathcal{L}_{K} (E;F)$ and $D(\mathcal{L}_{K} (E;F))\subset \mathcal{L}_{K} (M;N)$, the
operators
$$\tilde{C}:[S]\in \mathcal{L}(M;N)/\mathcal{L}_{K} (M;N)\rightarrow   [A_{2}\circ S\circ B_{1}]\in \mathcal{L}(E;F)/\mathcal{L}_{K} (E;F)$$
and
$$\tilde{D}:[T]\in \mathcal{L}(E;F)/\mathcal{L}_{K} (E;F)\rightarrow   [B_{2}\circ T\circ A_{1}]\in \mathcal{L}(M;N)/\mathcal{L}_{K} (M;N)$$
are well defined, and $\tilde{D}\circ \tilde{C}=I$, thus completing the proof.
\end{proof}

\begin{theorem}\label{proposicion 9.2}
 Let $1< p\leq q<\infty$. If $E$ and $F$ contain complemented subspaces isomorphic to $\ell_{p}$ and $\ell_{q}$, respectively, then $\mathcal{L}(E;F)/\mathcal{L}_{K} (E;F)$ does not have the approximation property.
\end{theorem}

\begin{proof}
By Theorem \ref{proposicion 9.1}    $\mathcal{L}(E;F)/\mathcal{L}_{K} (E;F)$ contains a complemented subspace isomorphic to $\mathcal{L}(\ell_{p};\ell_{q})/\mathcal{L}_{K} (\ell_{p};\ell_{q})$. Then the desired conclusion follows from Theorem \ref{proposicion 9}.
\end{proof}

By combining Theorem \ref{proposicion 9.2} and \cite[Lemma 2]{Pel} we obtain the following theorem.

\begin{theorem}\label{proposicion 9.3}
 Let $1< p\leq q<\infty$, and let $E$ and $F$ be closed infinite dimensional subspaces of  $\ell_{p}$ and $\ell_{q}$, respectively. Then $\mathcal{L}(E;F)/\mathcal{L}_{K} (E;F)$ does not have the approximation property.
\end{theorem}

\section{ Other examples of Banach spaces of linear operators  without the approximation property}

\begin{example}\label{ex1}
Let $U_{1}$ denote the universal space of Pelczynski (see, \cite[ Theorem 2.d.10]{LIN}).  $U_{1}$ is a Banach space with an unconditional
basis with the property that every  Banach space with an unconditional basis is isomorphic to a complemented subspace of $U_{1}$. Since
every $\ell_{p}$ $(1\leq p <\infty)$ has an unconditional basis, it follows that every $\ell_{p}$ $(1\leq p <\infty)$ is isomorphic to a complemented subspace of $U_{1}$. By Theorem \ref{corolario 1} none of the spaces $\mathcal{L}(U_{1};U_{1})$, $\mathcal{L}(U_{1}; \ell_{q})$ $(1<q<\infty)$ or $\mathcal{L}(\ell_{p};U_{1})$ $(1<p<\infty)$ have the approximation property.
\end{example}

\begin{definition}(see \cite[ p. 137]{LIN})

An Orlicz function M is a continuous convex nondecreasing function $M:[0,\infty)\rightarrow   \mathbb{R}$ such that $M(0)=0$ and $\lim\limits_{t \rightarrow \infty}M(t)=\infty$.
Let $$\ell_{M}=\bigg\{x=(\xi_{n})_{n=1}^{\infty}\subset \mathbb{K}: \sum_{n=1}^{\infty}M(|\xi_{n}/\rho)<\infty \ \  for \  some \ \  \rho> 0\bigg\}.$$
 Then $\ell_{M}$ is a Banach space for the norm
$$\|x\|=\inf\ \bigg\{\rho >0; \sum_{n=1}^{\infty}M(|\xi_{n}/\rho)\leq 1\bigg\}.$$
$\ell_{M}$ is called an Orlicz sequence space.
\end{definition}

\begin{example}\label{ex2}
Consider the Orlicz function $M_{p}(t)=t^{p}(1+|\log t|)$, with $1< p<\infty$. Then the Orlicz sequence space $\ell_{M_{p}}$ contains complemented subspaces isomorphic to  $\ell_{p}$ (see \cite[ p. 157]{LIN}). If $1< p\leq q<\infty$, then by Theorem \ref{corolario 1},
 $\mathcal{L}(\ell_{M_{p}};\ell_{M_{q}})$ does not have the approximation property.
\end{example}

\begin{definition}(see \cite[ p. 175]{LIN})

Let $1\leq p<\infty$ and let $w=\{w_{n}\}_{n=1}^{\infty}$ be a nonincreasing sequence of positive numbers such that $w_{1}=1$, $\lim\limits_{n \rightarrow \infty}w_{n}=0$ and $\displaystyle\sum_{n=1}^{\infty}w_{n}=\infty$. Let
$$d(w,p)= \bigg\{x=(\xi_{n})_{n=1}^{\infty}\subset \mathbb{K}: \|x\|= \sup_{\pi}\bigg(\sum_{n=1}^{\infty}|\xi_{\pi(n)}|^{p}w_{n}\bigg)^{1/p}<\infty\bigg\},$$
where $\pi$ ranges over all permutations  of $\mathbb{N}$. Then $d(w,p)$ is a Banach space, called a Lorentz sequence space.
\end{definition}

\begin{example}\label{ex3}
It follows from  \cite[ p. 177, Proposition 4.e.3]{LIN} that every closed infinite dimensional subspace of
 $d(w,p)$ contains a complemented subspace isomorphic to $\ell_{p}$. By Theorem \ref{corolario 1}, if $1< p\leq q<\infty$, and $E$ and $F$ are closed infinite dimensional subspaces of $d(w,p)$ and $d(w,q)$, respectively, then
 $\mathcal{L}(E; F)$  does not have the approximation property.
\end{example}

\section{Spaces of homogeneous polynomials without the approximation property}

Let $\mathcal{P}(^{n}E; F)$ denote the Banach space of all continuous $n$-homogeneous polynomials
from $E$ into $F$. We omit $F$ when $F = \mathbb{K}$. A polynomial $P \in \mathcal{P}(^{n}E; F)$ is called compact  if $P$ takes bounded subsets in $E$ into relatively compact subsets in $F$. Let $\mathcal{P}_{K}(^{n}E; F)$ denote the space of all compact  $n$-homogeneous polynomials from $E$ into $F$.
We refer to  \cite{SEAN} or \cite{J MUJICA} for background information on the
theory of polynomials on Banach spaces.
\bigskip

An important tool in this section is a linearization theorem due to Ryan \cite{RYAN}. We will use the following version of Ryan's linearization theorem, wich appeared in \cite{J}. Here $\tau_{c}$ denotes the compact-open topology.
\bigskip
\begin{theorem}\label{proposicion 10b}
For each Banach space $E$ and each $n\in\mathbb{N}$ let
$$Q(^{n}E)=(\mathcal{P}(^{n}E), \tau_{c})^{\prime},$$
with the norm induced by $\mathcal{P}(^{n}E)$, and let
$$\delta_{n}: x\in E\rightarrow \delta_{x}\in Q(^{n}E)$$
denote the evaluation mapping, that is, $\delta_{x}(P)=P(x)$ for every $x\in E$ and $P\in \mathcal{P}(^{n}E)$. Then $Q(^{n}E)$ is a Banach space and
$\delta_{n}\in \mathcal{P}(^{n}E; Q(^{n}E))$. The pair $(Q(^{n}E), \delta_{n})$ has the following universal property: for each Banach space $F$ and
each $P\in \mathcal{P}(^{n}E; F)$, there is a unique operator \  $T_{p}\in \mathcal{L}(Q(^{n}E); F)$ such that $T_{p}\circ \delta_{n}=P$. The mapping
$$P\in \mathcal{P}(^{n}E; F)\rightarrow  T_{p}\in \mathcal{L}(Q(^{n}E); F)$$
is an isometric isomorphism. Moreover $P \in\mathcal{P}_{K}(^{n}E; F)$ if only if \  $T_{p}\in \mathcal{L}_{K}(Q(^{n}E); F)$. Furthermore $Q(^{n}E)$ is isometrically isomorphic to $\hat{\otimes}_{n,s,\pi}E$, the completion of the space of n-symmetric tensors on $E$, with the
projective topology.

 \end{theorem}

\begin{theorem}\label{proposicion 10a}
If $E$ and $F$ contain complemented subspaces isomorphic to $M$ and $N$, respectively, then $\mathcal{P}(^{n}E;F)$ contains a complemented subspace
isomorphic to $\mathcal{P}(^{n}M;N)$.
 \end{theorem}

 \begin{proof}
 By hypothesis there are $A_{1}\in\mathcal{L}(M;E)$, $B_{1}\in\mathcal{L}(E;M)$, $A_{2}\in\mathcal{L}(N;F)$, $B_{2}\in\mathcal{L}(F;N)$ such that $B_{1}\circ A_{1}=I$ and $B_{2}\circ A_{2}=I$. Consider the operators
 $$C: P\in \mathcal{P}(^{n}M;N)\rightarrow A_{2}\circ P\circ B_{1}\in \mathcal{P}(^{n}E;F)$$
 and
 $$D: Q\in \mathcal{P}(^{n}E;F)\rightarrow B_{2}\circ Q\circ A_{1}\in \mathcal{P}(^{n}M;N).$$
 Then $D\circ C=I$ and the desired conclusion follows.
  \end{proof}

 \begin{corollary}\label{cori}
 If $E$ contains a complemented subspace isomorphic to $M$, then $\mathcal{P}(^{n}E)$ contains a complemented subspace isomorphic to $\mathcal{P}(^{n}M)$.
 \end{corollary}

 \begin{proof}
  Take $F=N=\mathbb{K}$ in Theorem \ref{proposicion 10a}.
  \end{proof}

The next theorem improves \cite[Theorem 3.2]{JORGE}.

\begin{theorem}\label{teo10}
Let $1<p<\infty$. If $E$ contains a complemented subspace isomorphic to $\ell_{p}$, then $\mathcal{P}(^{n}E)$ does not have the approximation property for every $n\geq p$.
\end{theorem}

\begin{proof}
By Corollary \ref{cori} $\mathcal{P}(^{n}E)$ contains a complemented subspace isomorphic to $\mathcal{P}(^{n}\ell_{p})$. Then the conclusion follows from
 \cite[Theorem 3.2]{JORGE}.
\end{proof}

Theorem \ref{teo10} can be used to produce many additional counterexamples. For instance, by combining Theorem \ref{teo10} and \cite[Lemma 2]{Pel}
we obtain the following result.

\begin{theorem}\label{teo11}
Let $1<p<\infty$ and let $E$ be a closed infinite dimensional subspace of $\ell_{p}$. Then $\mathcal{P}(^{n}E)$ does not have the approximation property
for every $n\geq p$.
\end{theorem}

In a similar way we may obtain scalar-valued polynomial versions of Theorem \ref{abc} and Examples \ref{ex1}, \ref{ex2} and \ref{ex3}.
We leave the details to the reader.

\begin{theorem}\label{teo12}
Let $1<p\leq q<\infty$. If $E$ and $F$ contain complemented subspaces isomorphic to $\ell_{p}$ and $\ell_{q}$, respectively, then $\mathcal{P}(^{n}E; F)$ does not have the approximation property
for every $n\geq 1$.
\end{theorem}

\begin{proof}

By \cite[ Proposition 5]{BLASCO}, $\mathcal{L}(E;F)$ is isomorphic to a complemented subspace of $\mathcal{P}(^{n}E;F)$. Then the desired conclusion follows from
Theorem \ref{corolario 1}.
\end{proof}

Theorem \ref{teo12} can be used to produce many additional counterexamples. For instance, by combining Theorem \ref{teo12} and \cite[Lemma 2]{Pel} we
obtain the following result.

\begin{theorem}\label{teo13}
Let $1<p\leq q<\infty$. and let $E$ and $F$ be closed infinite dimensional subspaces of $\ell_{p}$  and $\ell_{q}$, respectively. Then $\mathcal{P}(^{n}E; F)$ does not have the approximation property
for every $n\geq 1$.
\end{theorem}

In a similar way we may obtain vector-valued polynomial versions of Theorem \ref{abc} and Examples \ref{ex1}, \ref{ex2} and \ref{ex3}. We leave the details
to the reader.

\begin{theorem}\label{teo14}
 If $n<p\leq q<\infty$. Then $\mathcal{P}(^{n}\ell_{p}; \ell_{q})/\mathcal{P}_{K}(^{n}\ell_{p}; \ell_{q})$ does not have the approximation property.
\end{theorem}

\begin{proof}
By Theorem \ref{proposicion 10b} we can write
$$\mathcal{P}(^{n}\ell_{p}; \ell_{q})=\mathcal{L}(Q(^{n}\ell_{p}); \ell_{q})$$
and
$$\mathcal{P}_{K}(^{n}\ell_{p}; \ell_{q})=\mathcal{L}_{K}(Q(^{n}\ell_{p}); \ell_{q}).$$
We apply \cite[Theorem 2.4]{THREE}. By \cite[Lemma 1]{REMARKS} $\mathcal{L}_{K}(Q(^{n}\ell_{p}); \ell_{q})^{\perp}$ is a complemented subspace of
$\mathcal{L}(Q(^{n}\ell_{p}); \ell_{q})^{\prime}$. By \cite[Remark 3.3]{JORGE} $\mathcal{P}(^{n}\ell_{p})$ is a reflexive
Banach space with a Schauder basis. Hence $Q(^{n}\ell_{p})$ is also a reflexive Banach space with a Schauder basis. Then
by \cite[Corollary 16.69]{FABIAN} $\mathcal{L}_{K}(Q(^{n}\ell_{p}); \ell_{q})$ has a Schauder basis. If we assume that
$\mathcal{L}(Q(^{n}\ell_{p}); \ell_{q})/\mathcal{L}_{K}(Q(^{n}\ell_{p}); \ell_{q})$ has the approximation property,
then \cite[Theorem 2.4]{THREE} would imply that $\mathcal{L}(Q(^{n}\ell_{p}); \ell_{q})$ has the approximation property. But this
contradicts Theorem \ref{proposicion 9.2}, since $\ell_{p}=Q(^{1}\ell_{p})$ is a complemented subspace of $Q(^{n}\ell_{p})$, by
\cite[Theorem 3]{BLASCO}. This completes the proof.
\end{proof}

\begin{theorem}\label{teo20}
If $E$ and $F$ contain complemented subspaces isomorphic to $M$ and $N$, respectively, then
 $\mathcal{P}(^{n}E; F)/\mathcal{P}_{K}(^{n}E; F)$  contains a complemented subspace isomorphic to $\mathcal{P}(^{n}M; N)/\mathcal{P}_{K}(^{n}M; N)$.
\end{theorem}

\begin{proof}
By hypothesis there are $A_{1}\in\mathcal{L}(M;E)$, $B_{1}\in\mathcal{L}(E;M)$, $A_{2}\in\mathcal{L}(N;F)$, $B_{2}\in\mathcal{L}(F;N)$ such that $B_{1}\circ A_{1}=I$ and $B_{2}\circ A_{2}=I$. Let
$$C: \mathcal{P}(^{n}M;N)\rightarrow  \mathcal{P}(^{n}E;F)$$
 and
$$D: \mathcal{P}(^{n}E;F)\rightarrow \mathcal{P}(^{n}M;N)$$
be the operators from the proof of Theorem \ref{proposicion 10a}. Since $C(\mathcal{P}_{K}(^{n}M; N))\subset \mathcal{P}_{K}(^{n}E; F)$ and
$D(\mathcal{P}_{K}(^{n}E; F))\subset \mathcal{P}_{K}(^{n}M; N)$, the operators
$$\tilde{C}:[P]\in \mathcal{P}(^{n}M; N)/\mathcal{P}_{K}(^{n}M; N)\rightarrow [A_{2}\circ P\circ B_{1}]\in \mathcal{P}(^{n}E; F)/\mathcal{P}_{K}(^{n}E; F)$$
and
$$\tilde{D}:[Q]\in \mathcal{P}(^{n}E; F)/\mathcal{P}_{K}(^{n}E; F)\rightarrow [B_{2}\circ Q\circ A_{1}]\in \mathcal{P}(^{n}M; N)/\mathcal{P}_{K}(^{n}M; N)$$
are well-defined and $\tilde{D}\circ \tilde{C} =I$, thus completing the proof.

\end{proof}

\begin{theorem}\label{teo21}
Let $n<p\leq q<\infty$. If $E$ and $F$ contain complemented subspaces isomorphic to  $\ell_{p}$ and $\ell_{q}$, respectively,
then $\mathcal{P}(^{n}E; F)/ \mathcal{P}_{K}(^{n}E; F)$ does not have the approximation property.
\end{theorem}

\begin{proof}
By Theorem \ref{teo20}  $\mathcal{P}(^{n}E; F)/ \mathcal{P}_{K}(^{n}E; F)$ contains a complemented subspace isomorphic to $\mathcal{P}(^{n}\ell_{p}; \ell_{q})/ \mathcal{P}_{K}(^{n}\ell_{p}; \ell_{q})$. Thus the desired conclusion follows from Theorem \ref{teo14}.
\end{proof}
\bigskip

Theorem \ref{teo21} can be used to produce many additional counterexamples. For instance by combining Theorem \ref{teo21} and \cite[Lemma 2]{Pel} we obtain
the following theorem.

\begin{theorem}\label{teo22}
Let $n<p\leq q<\infty$, and let $E$ and $F$ be closed infinite dimensional subspaces of  $\ell_{p}$ and $\ell_{q}$, respectively.
Then $\mathcal{P}(^{n}E; F)/ \mathcal{P}_{K}(^{n}E; F)$ does not have the approximation property.
\end{theorem}

The interest in the study of the approximation property in spaces of homogeneous polynomials begun in $1976$ with a paper of Aron and Schottenloher
\cite{ARON}. They begun the study of the approximation property on the space $\mathcal{H}(E)$ of all holomorphic functions on $E$ under various topologies.
Among many other results they proved that $(\mathcal{H}(E),\tau_{w})$ has the approximation property if and only if $\mathcal{P}(^{n}E)$ has the
approximation property for every $n\in \mathbb{N}$. Here $\tau_{w}$ denotes the compact-ported topology introduced by Nachbin. They also proved
that $\mathcal{P}(^{n}\ell_{1})$ has the approximation property for every $n\in \mathbb{N}$. Ryan \cite{RYAN} proved that $\mathcal{P}(^{n}c_{0})$ has
a Schauder basis, and in particular has the approximation property, for every $n\in\mathbb{N}$.
Tsirelson \cite{T} constructed a reflexive Banach space $X$, with an unconditional Schauder basis, which contains no subspace isomorphic to any
$\ell_{p}$. By using a result of Alencar, Aron and Dineen \cite{AR}, Alencar \cite{AL} proved that $\mathcal{P}(^{n}X)$ has a Schauder basis,
and in particular has the approximation property, for every $n\in\mathbb{N}$. In a series of papers Dineen and Mujica \cite{D1} \cite{D2} \cite{D3}
have extended some of the results of Aron and Schottenloher \cite{ARON} to spaces of holomorphic functions defined on arbitrary open sets.


\begin{thebibliography}{99}                                                                                               %

\bibitem{AL}
\newblock R. Alencar,   {\it On reflexivity and basis for $\mathcal{P}(^{m}E)$,}  Proc. Roy. Irish Acad ., 85
(1985) 131- 138.


\bibitem{AR}
\newblock R. Alencar, R. Aron, S. Dineen,  {\it  A reflexive space of holomorphic functions
in infinitely many variables,} Proc. Amer. Math. Soc., 90 ( 1984) 407- 411.

\bibitem{ARON}
\newblock R.M.Aron, M. Schottenloher,  {\it Compact holomorphic mappings on Banach spaces and
the approximation property .} J. Funct. Anal. 21 (1976), 7–30.


\bibitem{BLASCO}
\newblock F. Blasco,  {\it Complementation of symmetric tensor products and polynomials, } Studia Math. 123 (1997) 165-173.

\bibitem{LORENTZ}
\newblock P. G. Casazza and B. Lin, {\it On symmetric basic sequences in Lorentz sequences spaces II}, Israel. J. Math. 17 (1974), 191-218.

\bibitem{Díaz}
\newblock J. C. Díaz, S. Dineen, {\it Polynomials on stable spaces}. Ark. Mat. 36 (1998), 87–96.


\bibitem{SEAN}
\newblock S. Dineen,  {\it Complex Analysis on Infinite Dimensional Spaces, } Springer, 1999.

\bibitem{D1}
\newblock S. Dineen, J.Mujica:  {\it The approximation property for spaces of holomorphic functions
on infinite dimensional spaces. I. } J. Approx. Theory 126 (2004), 141–156.


\bibitem{D2}
\newblock S. Dineen, J.Mujica:  {\it The approximation property for spaces of holomorphic functions
on infinite dimensional spaces. II. }J. Funct. Anal. 259 (2010), 545–560.

\bibitem{D3}
\newblock S. Dineen, J.Mujica:  {\it The approximation property for spaces of holomorphic functions
on infinite dimensional spaces. III. }Rev. R. Acad. Cienc. Exactas Fís. Nat., Ser. A Mat.,
RACSAM 106 (2012), 457–469.





\bibitem{JORGE}
\newblock S. Dineen, J. Mujica,  {\it Banach spaces of homogeneous polynomials without the approximation property, } Czechoslovak Math. J., 65 (140) (2015) 367-374.

\bibitem{ENFLO}
\newblock P. Enflo, {\it A counterexample to the approximation problem in Banach spaces}. Acta Math.
130 (1973), 309–317.


\bibitem{THREE}
\newblock G. Godefroy, P.D. Saphar,  {\it Three-space problems for the approximation properties, }  Proc. Amer. Math. Soc. 105, 70-75 (1989).

\bibitem{GR}
\newblock A. Grothendieck, {\it Produits Tensoriels Topologiques et Espaces Nucléaires}. Mem. Amer.
Math. Soc. 16 (1955), 140 pages. (In French.)


\bibitem{REMARKS}
\newblock J. Johnson,  {\it Remarks on Banach spaces of compact operators},  J. Funct. Anal. 32 (1979), 304-311.

\bibitem{LIN}
\newblock J. Lindenstrauss and L. Tzafriri,  {\it Classical Banach spaces I}, Springer, Berlin  1977.

\bibitem{LIND}
\newblock J. Lindenstrauss and L. Tzafriri,  {\it Classical Banach spaces II}, Springer, Berlin  1979.

\bibitem{FABIAN}
\newblock F. Marián, P. Habala, P. Hájek,  V. Montesinos, V. Zizler,  {\it Banach
Space Theory: The Basis for Linear and Nonlinear Analysis}. New York, Springer, 2011.



\bibitem{J MUJICA}
\newblock J. Mujica, {\it Complex Analysis in Banach Spaces. Holomorphic Functions and Domains of
Holomorphy in Finite and Infinite Dimensions.} North-Holland Math. Stud. 120. Notas
de Matemática 107, North-Holland, Amsterdam, 1986.

\bibitem{J}
\newblock J. M u j i c a, {\it Linearization of bounded holomorphic mappings on Banach spaces},
Trans. Amer. Math. Soc., 324 (1991) 867-887.


\bibitem{Pel}
\newblock A. Pelczy$\acute{n}$ski, {\it Projections in certain Banach spaces},  Stud. Math. 19 (1960), 209-228.

\bibitem{BANACH}
\newblock A. Pelczynski, C. Bessaga, {\it Some aspects of the present theory of Banach spaces, in: S. Banach, Theory of Linear Operations }, North- Holland Mathematical Library Vol. 38, North- Holland, Amsterdam, 1987.

\bibitem{Pitt}
\newblock H. Pitt, {\it A note on bilinear forms,} London Math. Soc., 11 (1936) 174- 180.


\bibitem{RYAN}
\newblock R. Ryan, {\it Applications of topological tensor products to infinite dimensional
holomorphy}, Ph. D. Thesis, Trinity College, Dublin 1980.

\bibitem{SZ}
\newblock A. Szankowski, {\it B(H) does not have the approximation property}. Acta Math. 147 (1981),
89–108.

\bibitem{T}
\newblock B. Tsirelson, {\it Not every Banach space contains an imbedding of $\ell_{p}$ or $c_{0}$}, Functional Anal. Appl., 9 (1974) 138- 141.
\end{thebibliography}
\end{document}